\documentclass[12pt,notitlepage,a4paper]{article}
\usepackage{amsmath,amssymb,amsthm}
\usepackage[arrow,matrix]{xy}
\theoremstyle{plain}
\numberwithin{equation}{section}
\newtheorem{thm}{Theorem}[section]

\newtheorem{prop}[thm]{Proposition}
\newtheorem*{conj}{Conjecture}
\theoremstyle{definition}
\newtheorem{df}[thm]{Definition}

\newtheorem{rmk}[thm]{Remark}
\newcommand{\mb}{\mathbb}
\newcommand{\mf}{\mathfrak}
\newcommand{\ml}{\mathcal}
\newcommand{\en}{{\rm End}}
\newcommand{\Q}{\mb{Q}}
\newcommand{\F}{\mb{F}}
\newcommand{\C}{\mb{C}}

\begin{document}
\author{Rin Sugiyama\\ \textit{\small Graduate School of Mathematics, Nagoya University,}\\ \textit{\small Furo-cho, Chikusa-ku, Nagoya 464-8602, Japan}\\ {\small e-mail: rin-sugiyama@math.nagoya-u.ac.jp}
}
\title{Remark on nondegeneracy of simple abelian varieties with many endomorphisms}
\maketitle

\begin{abstract}
We investigate a relationship between nondegeneracy of a simple abelian variety $A$ over an algebraic closure of $\mb{Q}$ and of its reduction $A_0$.
We prove that under some assumptions, nondegeneracy of $A$ implies nondegeneracy of $A_0$.
\end{abstract}



\section*{Introduction}
Let $A$ be an abelian variety over an algebraically closure $\Q^{\rm alg}$ of $\Q$ in $\C$.
In this paper, we say that $A$ is an abelian variety \textit{with many endomorphisms} if the reduced degree of the $\Q$-algebra $\en^0(A):=\en(A)\otimes\Q$ is equal to $2\dim A$
\footnote{The reduced degree of $\en^0(A)$ is always $\leq 2\dim A$.}.
This condition is equivalent to that $A$ is of CM-type.
An abelian variety $A$ over $\Q^{\rm alg}$ is said to be \textit{nondegenerate} if all the Hodge classes (see \S 1) on $A$ are generated by divisor classes in the Hodge ring of $A$.
If $A$ is nondegenerate,  then the Hodge conjecture holds for $A$.
We know that products of elliptic curves over $\C$ are nondegenerate (Tate \cite{Ta}, Murasaki \cite{M}, Imai \cite{I}, Murty \cite{Mu}).
However,  there are examples which is degenerate but the Hodge conjecture holds (cf. \cite{Ao} \cite{Sh}).
For other known results on the Hodge conjecture for abelian varieties, we refer to Gordon's article in Lewis's book \cite[Appendix B]{Le}.

Let $p$ be a prime number.
Let $\F$ be an algebraic closure of a finite field $\F_p$ with $p$-elements.
Let $\ell$ be a prime number different from $p$.
An abelian variety $A_0$ over $\F$ is said to be \textit{nondegenerate} if all the $\ell$-adic Tate classes (see \S 1) on $A_0$ are generated by divisor classes in the $\ell$-adic \'etale cohomology ring of $A_0$.
If $A_0$ is nondegenerate,  then the Tate conjecture holds for $A_0$.
Spiess \cite{Sp} proved that products of elliptic curves over finite fields are nodegenerate.
For certain abelian varieties over finite fields, nondegeneracy is known by Lenstra-Zarhin \cite{LZ}, Zarhin \cite{Za}, Kowaloski \cite{Ko}.
However, there are examples which is not nondegenerate but the Tate conjecture holds (\cite[Example 1.8]{Mi3}).
For other known results on the Tate conjecture, we refer to \cite{Ta3}.

Milne \cite[Theorem]{Mi2} proved that if the Hodge conjecture holds for all CM abelian varieties over $\C$, then the Tate conjecture holds for all abelian varieties over the algebraic closure of a finite field.
He furthermore studied a relationship between the Hodge conjecture for an abelian variety $A$ with many endomorphisms over $\Q^{\rm alg}$ and the Tate conjecture for the reduction $A_0/\F$ of $A$ at a prime $w$ of $\Q^{\rm alg}$ dividing $p$ (see Theorem \ref{Mil3}).
Here, we note that by a result of Serre--Tate \cite[Theorem 6]{SeTa}, one can consider the reduction of $A$.
However a relationship between nondegeneracy of $A/\Q^{\rm alg}$ and of $A_0/\F$ is not clear.
In this paper, we investigate a relationship between nondegeneracy of certain simple abelian variety with many endomorphisms over $\Q^{\rm alg}$ and of its reduction.
The following theorem is our main result.

\begin{thm}\label{mt}
Let $A$ be a simple abelian variety with many endomorphisms over $\Q^{\rm alg}$.

$(1)$ Assume that the CM-field $\en^0(A)$ is a abelian extension of $\Q$.
If all powers of $A$ are nondegenerate, then for any prime $w$ of $\Q^{\rm alg}$, all powers of a simple factor of the reduction of $A$ at $w$ are nondegenerate.

$(2)$ Let $w$ be a prime of $\Q^{\rm alg}$.
Let $A_0$ be the reduction of $A$ at $w$.
Assume that the restriction of $w$ to the Galois closure of the CM-field $\en^0(A)$ is unramified over $\Q$ and its absolute degree is one.
\begin{itemize}
\setlength{\parskip}{-3pt}
\item[{\rm (a)}] If the Hodge conjecture holds for all powers of $A$, then the Tate conjecture holds for all powers of $A_0$.

\item[{\rm (b)}] All powers of $A$ are nondegenerate if and only if all powers of $A_0$ are nondegenerate.
\end{itemize}
\end{thm}

Statement (2) of the theorem is almost a corollary of a result of Milne.
We prove this theorem, using a result of Milne (Theorems \ref{Mil1} and \ref{Mil3}) and a necessary and sufficient condition for nondegeneracy (Theorem \ref{ns2}, Theorem \ref{rs2}).
The key (Proposition \ref{keyp}) is to compare the conditions of nondegeneracy over $\C$ and $\F$ by a result of Shimura--Taniyama on the prime ideal decomposition of Frobenius endomorphism .

\medskip
This paper is organized as follows:
In section 1, we recall Milne's results \cite{Mi2,Mi3} on the Hodge conjecture and the Tate conjecture.
We also recall a necessary and sufficient condition for nondegeneracy of certain simple abelian varieties (Theorem \ref{ns2}, Theorem \ref{rs2}).
In section 2, we prove a key proposition (Proposition \ref{keyp}) for our main result.
Using the key proposition and results mentioned in section 1, we give a proof of Theorem \ref{mt}.
In the last section, using a result of Aoki \cite{Ao} we give an example of a degenerate simple abelian variety over $\F$ for which the Tate conjecture holds.
\bigskip

\noindent\textit{Notation.}

For an abelian variety $A$ with many endomorphisms over an algebraically closed field $k$, $\en^0(A)$ denotes $\en_k(A)\otimes \mb{Q}$, and $C(A)$ denotes the center of $\en^0_k(A)$.

For a finite \'etale $\mb{Q}$-algebra $E$, $\Sigma_E:={\rm Hom}(E,\mb{Q}^{\rm alg})$.
If $E$ is a field Galois over $\mb{Q}$, we identify $\Sigma_E$ with the Galois group ${\rm Gal}(E/\mb{Q})$.

For a finite set $S$, $\mb{Z}^S$ denotes the set of functions $f : S\rightarrow \mb{Z}$. 

An affine algebraic group is of multiplicative type if it is  commutative and its identity component is a torus.
For such a group $W$ over $\mb{Q}$, $\chi(W):={\rm Hom}(W_{\mb{Q}^{\rm alg}},\mb{G}_m)$ denotes the group of characters of $W$.

For a finite \'etale $\mb{Q}$-algebra $E$, $(\mb{G}_m)_{E/\mb{Q}}$ denotes the Weil restriction ${\rm Res}_{E/\mb{Q}}(\mb{G}_m)$ which is characterized by $\chi((\mb{G}_m)_{E/\mb{Q}})=\mb{Z}^{\Sigma_E}$.

\section{The Hodge conjecture and the Tate conjecture for abelian varieties}

We first recall a statement of conjectures.\medskip

\noindent\textit{The Hodge conjecture}. Let $A$ be an abelian variety of dimension $g$ over $\C$.
By $H_B^{*}(A,\Q)$ we denote the Betti cohomology of $A$.
For each integer $i$ with $0\leq i \leq g$, we define the space of \textit{the Hodge classes} of degree $i$ on $A$ as follows:
\[
H_{B}^{2i}(A,\Q)\cap H^i(A,\Omega^i).
\]
We know that the image of a cycle map is contained in the space of the Hodge classes.
A Hodge class is said to be  \textit{algebraic} if it belongs to the image of cycle map.
\begin{conj}
All Hodge classes on $A$ are algebraic.
\end{conj}
By the Lefschetz--Hodge theorem, all the Hodge classes of degree one are generated by divisor classes.
Therefore $A$ is nondegenerate if and only if all the Hodge classes on $A$ are generated by the Hodge classes of degree one.
\bigskip

\noindent\textit{The Tate conjecture}. Let $\F_p$ be a finite field with $p$-elements and let $\F$ be the algebraic closure of $\F_p$.
Let $A_1$ be an abelian variety of dimension $g$ over a finite subfield $\mb{F}_q$ of $\mb{F}$.
Let $A_0$ be the abelian variety $A_1\otimes_{\mb{F}_q}\mb{F}$ over $\mb{F}$.
By $H^{2i}(A_0,\mb{Q}_{\ell}(i))$ we denote the $\ell$-adic \'etale cohomology group of $A_0$.
For each integer $i$ with $0\leq i \leq g$, we define the space of \textit{the $\ell$-adic Tate classes} of degree $i$ on $A$ as follows:
$$\varinjlim_{L/\mb{F}_q}H^{2i}(A_0,\mb{Q}_{\ell}(i))^{{\rm Gal}(\mb{F}/L)}.$$
Here $L/\mb{F}_q$ runs over all finite extensions of $\mb{F}_q$.
We know that the image of the $\ell$-adic \'etale cycle map is contained in the space of the Tate classes.
A Tate class is said to be \textit{algebraic} if it belongs to the image of the $\ell$-adic \'etale cycle map.
\begin{conj} All Tate classes on $A_0$ are algebraic.
\end{conj}
This is conjectured by Tate \cite[Conjecture 1]{Ta}.
By a result of Tate \cite{Ta1}, we know that for any abelian variety $A_0$, all the Tate classes of degree one are generated by divisor classes on $A_0$.
Therefore $A_0$ is nondegenerate if and only if all the Tate classes on $A_0$ are generated by the Tate classes of degree one.


\subsection{Necessary and sufficient condition}

Let $A$ be an abelian variety over an algebraically closed field $k$ such that the reduced degree of $\en^0(A)$ is $2\dim A$.
In this case, $A$ is said to have many endomorphisms.
There are important algebraic groups of multiplicative type $L(A), M(A), MT(A)$ and $P(A)$ over $\Q$ attached to $A$. 
Using these groups, Milne gave a necessary and sufficient condition for the Hodge conjecture and the Tate conjecture for abelian varieties with many endomorphisms.

\begin{thm}[Milne \text{\cite[p.\ 14, Theorem]{Mi3}}]\label{Mil1}
$(1)$ Let $A$ be an abelian variety with many endomorphisms over an algebraically closed field $k$ of characteristic zero.
Then $MT(A)\subset M(A)\subset L(A)$, and
\begin{itemize}
\setlength{\parskip}{-5pt}
\item[{\rm (i)}] the Hodge conjecture holds for all powers of $A$ if and only if $MT(A)=M(A)${\rm ;}
\item[{\rm (ii)}] all powers of $A$ are nondegenerate if and only if $MT(A)=L(A)$. 
\end{itemize}

\noindent$(2)$ Let $A_0$ be an abelian variety over $\mb{F}$.
Then $P(A_0)\subset M(A_0)\subset L(A_0)$, and
\begin{itemize}
\setlength{\parskip}{-5pt}
\item[{\rm (i)}] the Tate conjecture holds for all powers of $A_0$ if and only if $P(A_0)=M(A_0)${\rm ;}
\item[{\rm (ii)}] all powers of $A_0$ are nondegenerate if and only if $P(A_0)=L(A_0)$.
\end{itemize} 
\end{thm}

For the relationship between the Hodge conjecture and the Tate conjecture, Milne proved the following:

\begin{thm}[Milne \cite{Mi3}]\label{Mil3}
Let $A$ be an abelian variety with many endomorphisms over $\Q^{\rm alg}$ and let $A_0$ be the reduction of $A$ at a prime of $\Q^{\rm alg}$. If the Hodge conjecture holds for all powers of $A$ and
\[
P(A_0)=L(A_0)\cap MT(A)\quad (\text{intersection inside $L(A)$}),
\]
then the Tate conjecture holds for all powers of $A_0$.
\end{thm}

In the rest of this subsection, we briefly recall the definitions of the groups $L, M, MT$ and $P$ associated to an abelian variety $A$ over $k$ (For more detail, see \cite{Mi1}, \cite{Mi2} and \cite{Mi3}), and we recall a necessary and sufficient condition for nondegeneracy for certain simple abelian varieties (Theorem \ref{ns2}, Theorem \ref{rs2}).

Let $A$ be an abelian variety with many endomorphisms over an algebraically closed field $k$.
Put $E:=\en^0(A)$.
Let $C(A)$ be the center of $E$.
We write $A^{\vee}$ for the dual abelian variety of $A$. A polarization $\lambda: A\rightarrow A^{\vee}$ of $A$ determines an involution  of $E$ which stabilizes $C(A)$.
The restriction of the involution to $C(A)$ is independent of the choice of $\lambda$.
By $\dagger$, we denote this restriction to $C(A)$.

\begin{df}[\text{\cite[4.3, 4.4]{Mi1}, \cite[p.\ 52--53]{Mi2}, \cite[A.3]{Mi3}}]
The \textit{Lefschetz group} $L(A)$ of $A$ is the algebraic group over $\mb{Q}$ such that
\begin{align*}
L(A)(R)=\{\alpha \in (C(A)\otimes R)^{\times} \ | \ \alpha \alpha^{\dagger} \in R^{\times}\}
\end{align*}
 for all $\mb{Q}$-algebras $R$.
\end{df}

In case that $k=\C$, we can describe $L(A)$ as a subgroup of $(\mb{G}_m)_{E/\mb{Q}}$ in terms of characters as follows (\cite[A.7]{Mi3}): $L(A)$ is a subgroup of $(\mb{G}_m)_{E/\mb{Q}}$ whose character group is 
\begin{align}\label{cl1}
\frac{\mb{Z}^{\Sigma_E}}{\{g \in \mb{Z}^{\Sigma_E}\ \ | \ g=\iota g \ \text{and} \ \sum g(\sigma)=0\}}.
\end{align}
Here $\iota g$ is a function sending an element $\sigma$ of $\Sigma_E$ to $g(\iota \sigma)$, and $\sum g(\sigma)$ denotes $\displaystyle \sum _{\sigma \in \Sigma_E}g(\sigma)$.\medskip

In case that $k=\F$, $L(A)$ is a subgroup of $(\mb{G}_m)_{C(A)/\mb{Q}}$ whose character group is 
\begin{align}\label{cl2}
\frac{\mb{Z}^{\Sigma_{C(A)}}}{\{g \in \mb{Z}^{\Sigma_{C(A)}}\ \ | \ g=\iota g \ \text{and} \ \sum g(\sigma)=0\}}.
\end{align}
\medskip
\begin{df}
Jannsen \cite{Ja} proved that the category of motives generated by abelian varieties over $\mb{F}$ with the algebraic cycles modulo numerical equivalence as the correspondences is Tannakian.
The group $M(A)$ is defined as the fundamental group of the Tannakian subcategory of this category generated by $A$ and the Tate object.
\end{df}

\begin{df}
When the characteristic of $k$ is zero, the \textit{Mumford--Tate group} $MT(A)$ is defined to be the largest algebraic subgroup of $L(A)$ fixing the Hodge classes on all powers of $A$. 
\end{df}

When $A$ is simple and the characteristic of $k$ is zero,  we describe a condition for which a character of $L(A)$ is trivial on $MT(A)$.
To give the condition, we introduced notion of CM-type.

Let $E$ be a CM-algebra.
A subset $\Phi$ of $\Sigma_E$ is called \textit{CM-type} of $E$ if $\Sigma_E=\Phi\cup\iota\Phi$ and $\Phi\cap\iota\Phi=\phi$.
Here $\iota$ is complex conjugation on $\mb{C}$.

When $E=\en^0(A)$, the action of $E$ on $\Gamma(A,\Omega^1)$ defines a CM-type of $E$.

Now assume that $A$ is simple.
Let $\Phi$ be the CM-type of the CM-field $\en^0(A)$.
A character $g$ of $L(A)$ is trivial on $MT(A)$ if and only if 
\begin{align}\label{cmt}
\sum_{\sigma\in \Phi}g(\tau \circ \sigma)=0
\end{align}
for all $\tau \in {\rm Gal}(\Q^{\rm alg}/\Q)$.\medskip

For nondegeneracy of certain abelian varieties, the following result is known:

\begin{thm}\label{ns2}
Let $A$ be a simple abelian variety with many endomorphisms over $\Q^{\rm alg}$.
Let $\Phi$ be the CM-type of the CM-field $E:=\en^0(A)$ defined by the action of $E$ on $\Gamma(A,\Omega^1)$.
Assume that $E$ is a abelian extension over $\Q$ with its Galois group $G$.
Then all powers of $A$ are nondegenerate if and only if
\[
\sum_{\sigma \in \Phi}\chi(\sigma)\neq0
\]
for any character $\chi$ of $G$ such that $\chi(\iota)=-1$.
\end{thm}
For a proof of the theorem, see \cite{Ku}.

\begin{df}[\text{\cite[\S 4]{Mi2}, \cite[A.7]{Mi3}}]
Let $k$ be the algebraic closure $\F$ of a finite filed $\F_p$.
Let $A_1$ be a model of $A$ and let $\pi_1$ be the Frobenius endomorphism of $A_1$.
Then the group $P(A)$ is the smallest algebraic subgroup of $L(A)$ containing some power of $\pi_1$. It is independent of the choice of $A_1$.
\end{df}
When $A/\F$ is simple, we describe a condition for which a character of $L(A)$ is trivial on $P(A)$.
To give the condition, we introduce some notion about Weil numbers.

A \textit{Weil $q$-number} of weight $i$ is an algebraic number $\alpha$ such that $q^N\alpha$ is an algebraic integer for some $N$ and the complex absolute value $|\sigma(\alpha)|$ is $q^{i/2}$, for all embeddings $\sigma :\mb{Q}[\alpha]\rightarrow\mb{C}$.
We know that $\pi_1$ is a Weil $q$-number of weight one.
Then $\pi_1$ is a unit at all primes of $\mb{Q}[\pi_1]$ not dividing $p$.
We define the {\it slope function} $s_{\pi_1}$ of $\pi_1$ as follows: for any prime $\mf{p}$ dividing $p$ of a field containing $\pi_1$,
\begin{align}\label{sf}
s_{\pi}(\mf{p})=\frac{{\rm ord}_{\mf{p}}(\pi_1)}{{\rm ord}_{\mf{p}}(q)}.
\end{align}
The slope function determines a Weil $q$-number up to a root of unity.
From the definition of Weil numbers, $s_{\pi_1}(\mf{p})+s_{\pi_1}(\iota \mf{p})=1$.

We define a {\it Weil germ} to be an equivalent class
\footnote{Let $\pi$ be a Weil $p^f$-number and let $\pi^{\prime}$ be a Weil $p^{f^{\prime}}$-number. We say $\pi$ and $\pi^{\prime}$ are {\it equivalent} if $\pi^{f^{\prime}}={\pi^{\prime}}^f\cdot \zeta$ for some root of unity $\zeta$.}
of Weil numbers.
For a Weil germ $\pi$, the slope function of $\pi$ are the slope function (see \eqref{sf}) of any representative of $\pi$.\medskip

Now assume that $A/\F$ is simple.
Let $\pi_A$ denote the germ represented by $\pi_1$.
Milne's result on the character of $P(A)$ is the following (\cite[A.7]{Mi3}): let $g$ be a character of $L(A)$.
Then $g$ is trivial on $P(A)$ if and only if
\begin{align}\label{cp}
\sum_{\ \ \sigma \in \Sigma_{C(A)}} g(\sigma) s_{\sigma \pi_A}(\mf{p})=0.
\end{align}
for all primes $\mf{p}$ dividing $p$ of a field containing all conjugates $\sigma (\pi_1)$.\bigskip

Using Theorem \ref{Mil1} and \eqref{cl2} \eqref{cp}, we obtain the following: 

\begin{thm}[\cite{RS}]\label{rs2}
Let $A_0$ be a simple abelian variety over $\F$.
Assume that $C(A_0)$ is abelian extension of $\Q$ with its Galois group $G_0$.
Let $\mf{p}$ be a prime of $C(A_0)$ dividing $p$.
Then any power of $A_0$ is nondegenerate if and only if 
\[
\sum_{\sigma \in G_0}s_{\pi}(\sigma \mf{p})\chi(\sigma)\neq0
\]
for any character $\chi$ of $G_0$ such that $\chi(\iota)=-1$.
\end{thm}

This is an analogous result to Theorem \ref{ns2} for simple abelian varieties over $\F$.

\section{Proof of the main theorem}

We prove Theorem \ref{mt} using the theorems mention in the previous section.
We first fix the notation:
\begin{itemize}
\setlength{\parskip}{-5pt}
\item[$A$] : a simple abelian variety with many endomorphisms over $\Q^{\rm alg}$

\item[$w$] : a prime of $\Q^{\rm alg}$ dividing $p$

\item[$A_0$] : a simple factor of the reduction of $A$ at $w$
\bigskip

\item[$E$] : the CM-field $\en^0(A)$

\item[$\Phi$] : the CM-type of $E$, 

\item[$\varphi$] : the characteristic function of $\Phi$
\medskip

\item [$E_0$] : the center of $\en^0(A_0)$\quad  ($E_0$ is a subfield of $E$)

\item [$\pi$] : the Weil germ attached to $A_0$
\bigskip

\item[$K/\Q$]: a finite Galois extension which include all conjugate of $E$

\item[$G$] $:={\rm Gal}(K/\Q)$

\item[$\mf{p}$] : the restriction of $w$ to $K$

\item[$G_{\mf{p}}$] : the decomposition group of $\mf{p}$ in $K$
\end{itemize}\medskip

The following proposition is a key in a proof of our main result.
\begin{prop}\label{keyp}
Let the notation as above.
Then for any $\sigma \in \Sigma_{E_0}$ and any $\tau \in G$,
\[
s_{\sigma \pi}(\tau\mf{p})=\frac{1}{|G_{{\mf{p}}}|}\sum_{h \in G_{{\mf{p}}}}\varphi(h\tau^{-1}\circ\sigma).
\]
\end{prop}

\begin{proof}
We identify $\Sigma_E$ with ${\rm Hom}_{\Q}(E,K)$.
Let $\sigma \in \Sigma_{E_0}$ and $\tau \in G$.
Let $\tilde{\sigma}\in G$ be a lift of $\sigma$.
Since $E_0$ is equal to the smallest subfield of $\Q^{\rm alg}$ containing a representative of $\pi$, we have $s_{\sigma \pi}(\tau\mf{p})=s_{\tilde{\sigma} \pi}(\tau\mf{p})$.
Therefore we may fix the lift $\tilde{\sigma}\in G$ for each $\sigma \in \Sigma_{E_0}$.
Then we have $s_{\tilde{\sigma} \pi}(\tau\mf{p})=s_{\pi}({\tilde{\sigma}}^{-1}\tau\mf{p})$.
By a theorem of Shimura--Taniyama (see Tate \cite[Lemma 5]{TaH}), $s_{\pi}({\tilde{\sigma}}^{-1}\tau\mf{p})$ is given as follows
\begin{align*}
s_{\pi}({\tilde{\sigma}}^{-1}\tau\mf{p})
=\frac{|\Phi({\tilde{\sigma}}^{-1}\tau\mf{p})|}{|\Sigma_E({\tilde{\sigma}}^{-1}\tau\mf{p})|}
\end{align*}
where 
\begin{align*}
\Sigma_E({\tilde{\sigma}}^{-1}\tau\mf{p})&:=\{f \in \Sigma_E \ | \ {\tilde{\sigma}}^{-1}\tau\mf{p}=f^{-1}\mf{p} \ \cdots (*)\}\\
\Phi({\tilde{\sigma}}^{-1}\tau\mf{p})&:=\Phi\cap\Sigma_E({\tilde{\sigma}}^{-1}\tau\mf{p}).
\end{align*}
Here $(*)$ means that for any $x \in E$, 
\[
v_{\tau\mf{p}}({\tilde{\sigma}}(x))=v_{\mf{p}}(f(x)).
\]
Now we consider condition $(*)$.
Let $\tilde{f}\in G$ be a lift of $f$.
Then we have the following equivalences
\begin{align*}
\text{condition $(*)$}
&\Longleftrightarrow \text{$\tilde{\sigma}^{-1}\tau\mf{p}$ and $\tilde{f}^{-1}\mf{p}$ lie over the same prime of $E$}\\
&\Longleftrightarrow \tilde{\sigma}^{-1}\tau\mf{p}=\eta\tilde{f}^{-1}\mf{p}\ \ \text{for some $\eta \in {\rm Gal}(K/E)$}\\
&\Longleftrightarrow \tilde{f}\eta^{-1}\tilde{\sigma}^{-1}\tau \in G_{\mf{p}} \ \ \text{for some $\eta \in {\rm Gal}(K/E)$}\\
&\Longleftrightarrow \tilde{f} \in G_{\mf{p}}\tau^{-1}\tilde{\sigma} {\rm Gal}(K/E)
\end{align*}

If $\tilde{f} \in G_{\mf{p}}\tau^{-1}\tilde{\sigma} {\rm Gal}(K/E)$, then any lifts of $f$ are also in $G_{\mf{p}}\tau^{-1}\tilde{\sigma} {\rm Gal}(K/E)$.
Hence the property that $\tilde{f}$ belongs to $G_{\mf{p}}\tau^{-1}\tilde{\sigma} {\rm Gal}(K/E)$ is independent of the choice of the lift of $f$.

For $h_1,h_2\in G_{\mf{p}}$ and $\eta_1, \eta_2 \in {\rm Gal}(K/E)$, we have
\begin{align*}
(h_1\tau^{-1}\tilde{\sigma}\eta_1)^{-1}(h_2\tau^{-1}\tilde{\sigma}\eta_2) \in {\rm Gal}(K/E)
&\Longleftrightarrow \tilde{\sigma}^{-1}\tau h_1^{-1}h_2\tau^{-1}\tilde{\sigma} \in {\rm Gal}(K/E)\\
&\Longleftrightarrow h_1^{-1}h_2 \in {\rm Gal}(K/\tau^{-1}\tilde{\sigma}(E)).
\end{align*}
From the above argument, we have
\begin{align*}
|\Sigma_E({\tilde{\sigma}}^{-1}\tau\mf{p})|
&=|G_{\mf{p}}/G_{\mf{p}}\cap{\rm Gal}(K/\tau^{-1}\tilde{\sigma}(E))|\\
&=\frac{|G_{\mf{p}}|}{|G_{\mf{p}}\cap{\rm Gal}(K/\tau^{-1}\tilde{\sigma}(E))|}
\end{align*}\medskip

Next we calculate $|\Phi({\tilde{\sigma}}^{-1}\tau\mf{p})|$.
Since $h\tau^{-1}\tilde{\sigma}(x)=\tau^{-1}\tilde{\sigma}(x)$ for any $h \in {\rm Gal}(K/\tau\tilde{\sigma}(E))$ and for any $x \in E$, we have
\[
\varphi(h\tau^{-1}\circ\sigma)=\varphi(\tau^{-1}\circ\sigma).
\]
Therefore we have
\[
|\Phi({\tilde{\sigma}}^{-1}\tau{\mf{p}})|=\frac{1}{|G_{{\mf{p}}}\cap{\rm Gal}(K/\tau^{-1}\tilde{\sigma}(E))|}\sum_{h \in G_{{\mf{p}}}}\varphi(h\tau^{-1}\circ\sigma).
\]
Hence we have
\[
s_{\sigma \pi}(\tau\mf{p})=\frac{1}{|G_{{\mf{p}}}|}\sum_{h \in G_{{\mf{p}}}}\varphi(h\tau^{-1}\circ\sigma).
\]
\end{proof}

Before starting the proof of (1) of Theorem \ref{mt}, we prepare some notation.
By the assumption that $E$ is abelian over $\Q$, we may take $K=E$.
We write $G$ for the Galois group of $E/\Q$ and $G_0$ for the Galois group of $E_0/\Q$.
We identify $\Sigma_E$ (resp. $\Sigma_{E_0}$) with $G$ (resp. $G_0$). 
Then there is an exact sequence of finite abelian groups
\[
1\longrightarrow G_1\longrightarrow G\longrightarrow G_0\longrightarrow 1,
\]
where $G_1={\rm Gal}(E/E_0)$.\medskip

We define the subgroup $\hat{G^{-}}$ of the character group of $G$ as follows:
\[
\hat{G^{-}}:=\{ \chi :G\longrightarrow \C^{\times}\ | \ \chi(\iota)=-1\}.
\]
Here $\iota \in G$ is the complex conjugation.
Similarly to $\hat{G^{-}}$, we define the subgroup $\hat{G_0^{-}}$ of the character group of $G_0$.
Since $G_0$ is a quotient group of $G$, we consider $\hat{G_0^{-}}$ as the subgroup of $\hat{G^{-}}$:
\[
\hat{G_0^{-}}=\{ \chi \in \hat{G^{-}} \ | \ \chi(G_1)=1\}.
\]

\begin{rmk}\label{rk1}
Since $p$ is completely decomposed in $E_0$ (cf. \cite[Proposition 3.5]{RS} ), the decomposition group $G_{\mf{p}}$ of $\mf{p}$ is contained in $G_1$.
If $\mf{p}$ is unramified and its absolute degree is one, then $E=E_0$ and hence $G=G_0$.
\end{rmk}

\renewcommand{\proofname}{\bf Proof of (1) of Theorem \ref{mt}}
\begin{proof}
Let $\mf{p}_0$ be the prime $\mf{p}\cap E_0$ of $E_0$.
By Proposition \ref{keyp}, for any $\chi\in \hat{G_0^{-}}$ we have
\begin{align*}
\sum_{\sigma \in G_0}s_{\pi}(\sigma \mf{p}_0)\chi(\sigma)
&=\frac{1}{|G_1|}\sum_{\sigma \in G}s_{\pi}(\sigma \mf{p})\chi(\sigma)\\
&=\frac{1}{|G_1|}\sum_{\sigma \in G}\frac{1}{|G_{{\mf{p}}}|}\sum_{h \in G_{{\mf{p}}}}\varphi(h\sigma^{-1})\chi(\sigma)\\
&=\frac{1}{|G_1|\cdot |G_{{\mf{p}}}|}\sum_{h \in G_{\mf{p}}}\sum_{\sigma \in G}\varphi(h\sigma^{-1})\chi(\sigma)\\
&=\frac{1}{|G_1|}\sum_{\sigma \in G}\varphi(\sigma^{-1})\chi(\sigma)\\
&=\frac{1}{|G_1|}\sum_{\sigma \in \Phi}\bar{\chi}(\sigma).
\end{align*}
From this,  we obtain that for any $\chi\in \hat{G_0^{-}}\subset \hat{G^{-}}$,
\[
\sum_{\sigma \in G_0}s_{\pi}(\sigma \mf{p}_0)\chi(\sigma)\neq0\quad
\text{if and only if}\quad
\sum_{\sigma \in \Phi}\bar{\chi}(\sigma)\neq0.
\]
Therefore the assertion follows from Theorem \ref{ns2} and Theorem \ref{rs2}.
\end{proof}\bigskip

\renewcommand{\proofname}{\bf Proof of (2) of Theorem \ref{mt}}
\begin{proof}
To prove the assertion, by Theorem \ref{Mil1} and Theorem \ref{Mil3}, it suffices to show that $L(A)=L(A_0)$ and $MT(A)=P(A_0)$.

We first show that $L(A)=L(A_0)$.
By the assumption that $\mf{p}$ is unramified and its absolute degree is one, we obtain that $E=E_0$ from a result of Shimura--Taniyama \cite[p.\,100, Theorem 2]{ST}.
By the description \eqref{cl1} \eqref{cl2} of the character group of $L(A)$ and $L(A_0)$, we have $L(A)=L(A_0)$.

Next we show that $MT(A)=P(A_0)$.
We easily see that the condition \eqref{cmt} of triviality on $MT(A)$ of a character of $L(A)$ is described in terms of the characteristic function $\varphi$ of the CM-type $\Phi$ as follows: for all $\tau \in {\rm Gal}(\Q^{\rm alg}/\Q)$,
\begin{align}\label{rcmt}
\sum_{\sigma\in \Sigma_E}\varphi(\tau^{-1} \circ \sigma)g(\sigma)=0.
\end{align}

On the other hand, the assumption on $\mf{p}$ implies that $G_{\mf{p}}=1$.
Therefore, by Proposition \ref{keyp}, we obtain that for all $\tau \in {\rm Gal}(\Q^{\rm alg}/\Q)$,
\[
s_{\sigma \pi}(\tau \mf{p})=\varphi(\tau^{-1}\circ \sigma).
\]
From this equation and the equality $E=E_0$, the condition \eqref{cmt} of triviality on $P(A_0)$ of a character of $L(A)(=L(A_0))$ is described as follows: for all $\tau \in {\rm Gal}(\Q^{\rm alg}/\Q)$,
\begin{align}\label{rcp}
\sum_{\ \ \sigma \in \Sigma_{E_0}} g(\sigma) s_{\sigma \pi}(\tau \mf{p})
=\sum_{\sigma\in \Sigma_E}\varphi(\tau^{-1} \circ \sigma)g(\sigma)=0.
\end{align}
Since $L(A)=L(A_0)$ and conditions \eqref{rcmt} \eqref{rcp} are coincide, we obtain that $MT(A)=P(A_0)$.
This completes the proof.
\end{proof}

\section{Example}
From Theorem \ref{mt} and a result of Aoki \cite{Ao} on CM abelian varieties of Fermat type, we obtain examples of a simple \textit{degenerate} abelian variety $A_0$ over $\F$ for which the Tate conjecture holds. Here we give a such example.

Let $m=27$ and let $\alpha=(1,9,17)$.
Here $\alpha$ is an element of the set $\ml{A}_m^1$ defined as follows:
\[
\ml{A}_m^1:=\{\alpha=(a_0,a_1,a_2)\in (\mb{Z}/m\mb{Z})^3\  | \ a_i\not\equiv0 \ ({\rm mod}\ m), a_0+a_1+a_2\equiv0\ ({\rm mod}\ m)\}.
\]
We define a subset $\Phi_{\alpha}$ of $\mb{Z}/m\mb{Z}$ as
\[
\Phi_{\alpha}:=\{ t \in \mb{Z}/m\mb{Z} \ |\  \langle ta_0\rangle+\langle ta_1\rangle+\langle ta_2\rangle=m \}
\]
where for any $c \in \mb{Z}/m\mb{Z}$ we denote by $\langle c\rangle$ the least natural number such that $\langle c\rangle\equiv c \mod m$.
Then let $A=A_{\alpha}$ be a simple abelian variety with CM-type $(\Q(\mu_m), \Phi_{\alpha})$.
Then by a result of Aoki \cite[Theorem 2.1]{Ao}, $A$ is degenerate and the Hodge conjecture holds for all powers of $A$.

On the other hand, let $A_0$ be a simple factor of the reduction of $A$ at a prime $w$ of $\Q^{\rm alg}$ dividing a prime $p$.
By Theorem \ref{mt}\,(2),  we see that if $p \equiv 1 \mod m$, then $A_0$ is degenerate and the Tate conjecture holds for all powers of $A_0$.
Furthermore, using Theorem \ref{rs2}, one can see that  if $p^9 \equiv 1 \mod m$ then all powers of $A_0$ are nondegenerate. 

\bigskip

\noindent\textbf{Acknowledgements} \
The author expresses his gratitude to Professors Thomas Geisser, Kanetomo Sato and Hiromichi Yanai for many helpful suggestions and comments.


\end{document}